\documentclass[a4paper,reqno]{amsart}

\usepackage[utf8]{inputenc}
\usepackage{mathtools}

\newtheorem{theorem}{Theorem}
\newtheorem{lemma}[theorem]{Lemma}
\theoremstyle{remark}

\theoremstyle{definition}
\newtheorem{Def}[theorem]{Definition}
\usepackage{hyperref}
\title[A short note on the Radon-Riesz property]{A short note on the Radon-Riesz property for continuous Banach space valued functions}

\author{Arne Roggensack}

\begin{document}

\begin{abstract}
We present a generalization of the Radon-Riesz property to sequences of continuous functions with values in uniformly convex and uniformly smooth Banach spaces.
\end{abstract}
\date{\today}
\maketitle

\section{Introduction}
In many analytical questions, the space of continuous functions with values in a Banach space plays a crucial role. Naturally, it arises the question how to decide whether a sequence is converging or not. In certain Banach spaces, the Radon-Riesz property (also called Kadec–Klee property or property (H)) is a useful tool in order to determine that. It was first proved for $L^p(\Omega)$ by Radon \cite{radon1913} and Riesz \cite{riesz1928}. In this note, we give a generalization of this property to continuous Banach space valued functions for uniformly convex and uniformly smooth Banach spaces.

The well-known Radon-Riesz property ensures the strong convergence of a weakly convergent sequence $x_k\rightharpoonup x$, if the sequence of the norm $\|x_k\|$ is converging to $\|x\|$. There are different sufficient characterizations known in order to determine if a Banach space has this property. A classical result states that uniformly convex Banach spaces have the Radon-Riesz property. Details on this result can be found in \cite{megginson1998}.

The theorem stated in this note can e.g.\ be applied to proof the sequential continuity of the solution operator of the transport equation (see \cite{RoggensackDiss.2014} or \cite{RoggensackTransport.2014}).

\section{Radon-Riesz property}
We first recall the notion of uniform convexity and uniform smoothness.
\begin{Def}\label{def:uniform_convex}
	A normed vector space $V$ is called \emph{uniformly convex} if for all $\varepsilon >0$ there exists $\delta>0$ such that for all $x,y\in V$ with $\|x\|=\|y\|\leq 1$ the following implication is true:
	\begin{equation*}
	\|x-y\|\geq \varepsilon\Rightarrow \left\|\frac{x+y}2\right\|\leq 1-\delta.
	\end{equation*}
\end{Def}
\begin{Def}
	A normed vector space $V$ is called \emph{uniformly smooth} if for all $\varepsilon >0$ there exists $\delta>0$ such that it holds 
	\begin{equation*}
	\|x+y\|+\|x-y\|\leq 2+\varepsilon\|y\|
	\end{equation*}
	for all $x,y\in V$ with $\|x\|=1$ and $\|y\|\leq \delta$.
\end{Def}
It is well-known that $V$ is uniformly smooth if and only if $V'$ is uniformly convex.

The proof of the desired generalization follows the same lines as the proof of the original property for uniform convex spaces (see e.g.\ \cite{megginson1998}). We begin with the following auxiliary lemma ensuring the existence of a certain continuous function with values in the dual space.
\begin{lemma}
	\label{lemma:uniform_convex}
	Let $V$ be a uniform smooth Banach space and let be $f\in C([0,T],V)$. Then, there exists for each $\varepsilon>0$ a function $\varphi\in C([0,T],V')$ with $\langle \varphi(t),f(t)\rangle=\|f(t)\|_V$ and $\|\varphi(t)\|_{V'}=1$ for all $t$ with $\|f(t)\|_V\geq\varepsilon$.
\end{lemma}
\begin{proof}
	Let be $M=\{t\in[0,T]|\|f(t)\|_V\geq \varepsilon\}$.
	For each $t\in M$, we use the Hahn-Banach theorem to find a vector $\varphi(t)\in V'$ with $\langle \varphi(t),f(t)\rangle=\|f(t)\|_V$ and $\|\varphi(t)\|_{V'}=1$. We will show the continuity of $\varphi$ on the closed set $M$. Afterwards, Dugundji's extension theorem \cite{Dugundji.1951} yields the existence of a continuous extension of $\varphi$ on $[0,T]$.
	
	To show the continuity of $\varphi$ on $M$ with values in $V'$, let be $\tilde\varepsilon>0$. Since $V$ is uniformly smooth, the dual space $V'$ is uniformly convex and we can choose $\delta>0$ like in the Definition~\ref{def:uniform_convex} of the uniform convexity of $V'$. Because of the continuity of $\frac{f(t)}{\|f(t)\|_V}$ with values in $V$, we find $\gamma>0$ such that 
	\begin{equation*}
	\left\|\frac{f(s)}{\|f(s)\|_V}-\frac{f(t)}{\|f(t)\|_V}\right\|_V<2\delta
	\end{equation*}
	holds for all $s,t\in M$ with $|s-t|<\gamma$. Then, we compute
	\begin{align*}
	&\left\|\frac{\varphi(t)+\varphi(s)}{2}\right\|_{V'}\\
	&\quad=\frac 12\sup\limits_{\substack{v\in V\\ \|v\|\leq 1}}\left|\langle\varphi(t),v\rangle+\langle\varphi(s),v\rangle\right|\\
	&\quad\geq \frac 12\left|\left\langle \varphi(t),\frac{f(t)}{2\|f(t)\|_V}+\frac{f(s)}{2\|f(s)\|_V}\right\rangle+\left\langle \varphi(s),\frac{f(t)}{2\|f(t)\|_V}+\frac{f(s)}{2\|f(s)\|_V}\right\rangle\right|\\
	&\quad=\frac 12\left|\frac{\langle\varphi(t),f(t)\rangle}{\|f(t)\|_V}+\frac{\langle\varphi(s),f(s)\rangle}{\|f(s)\|_V}+\frac 12\left\langle\varphi(t),\frac{f(s)}{\|f(s)\|_V}-\frac{f(t)}{\|f(t)\|_V}\right\rangle\right.\\
	&\quad\qquad\left.+\frac 12\left\langle\varphi(s),\frac{f(t)}{\|f(t)\|_V}-\frac{f(s)}{\|f(s)\|_V}\right\rangle\right|\\
	&\quad\geq \frac 12\left(2-\frac 12\left|\left\langle\varphi(t),\frac{f(s)}{\|f(s)\|_V}-\frac{f(t)}{\|f(t)\|_V}\right\rangle\right|\right.\\
	&\quad\qquad\left.-\frac 12\left|\left\langle\varphi(s),\frac{f(t)}{\|f(t)\|_V}-\frac{f(s)}{\|f(s)\|_V}\right\rangle\right|\right)\\
	&\quad\geq 1-\frac 12\left\|\frac{f(s)}{\|f(s)\|_V}-\frac{f(t)}{\|f(t)\|_V}\right\|_V\\
	&\quad>1-\delta.
	\end{align*}
	Using the uniform convexity of the dual space $V'$ we conclude
	\begin{equation*}
	\|\varphi(t)-\varphi(s)\|_{V'}<\tilde\varepsilon,
	\end{equation*}
	which proves the desired continuity of $\varphi$.
\end{proof}
With this preliminary result, we are able to prove the main result of this note concerning the following generalization of the Radon-Riesz property for uniformly convex and uniformly smooth Banach spaces.
\begin{theorem}\label{lemma:gleichmaessig}
	Let $V$ be a uniformly convex and uniformly smooth Banach space. Let $(f_n)_n\subset C([0,T],V)$ be a sequence and $f\in C([0,T],V)$ a function such that the pair $((f_n)_n,f)$ 
	fulfils the following two properties:
	\begin{enumerate}
		\item $\|f_n(t)\|_V$ converges uniformly to $\|f(t)\|_V$ and
		\item $\langle\varphi(t),f_n(t)\rangle$ converges uniformly to $\langle\varphi(t),f(t)\rangle$ for all $\varphi \in C([0,T],V')$.
	\end{enumerate} 
	Then, $f_n$ converges  to $f$ in the norm of $C([0,T],V)$, i.e.\
	\begin{equation*}
	\lim\limits_{n\to \infty}\sup\limits_{t\in[0,T]}\|f_n(t)-f(t)\|_V=0.
	\end{equation*}
\end{theorem}
\begin{proof}
	If $f(t)=0$ for all $t$, the statement is clearly true due to the uniform convergence of the norm. Thus, let be $f\neq 0$.
	We suppose that $f_n$ does not converge uniformly to $f$ in  order to prove the theorem by contradiction. Then, there exists a subsequence, also denoted by $(f_n)_n$, such that there is an $\varepsilon>0$ with
	\begin{equation}
	\label{eq:not_converge}
	\|f_n-f\|_{C([0,T],V)}\geq \varepsilon
	\end{equation}
	for all $n$. 
	We choose $\delta\in (0,1]$ as in the Definition~\ref{def:uniform_convex} of the uniform convex Banach space $V$ for $\tilde{\varepsilon}=\frac{\varepsilon}{2\|f\|_{C([0,T],V)}}$ and $N_1$ such that it holds
	\begin{equation*}
	\Bigl|\|f_n(t)\|_V-\|f(t)\|_V\Bigr|< \frac{\varepsilon\delta}{3}
	\end{equation*}
	 for all $t$ and all $n\geq N_1$. Using Lemma~\ref{lemma:uniform_convex}, we find a function $\varphi\in C([0,T],V')$  with $\|\varphi(t)\|_{V'}=1$ and $\langle\varphi(t),f(t)\rangle=\|f(t)\|_{V}$ for all $t\in [0,T]$ with $\|f(t)\|_V\geq \frac{\varepsilon}3$.
	 
	Because of the uniform convergence of $\langle\varphi(t),f_n(t)\rangle$, there is also an integer $N_2$ such that it holds
	\begin{equation*}
	\sup\limits_{t\in [0,T]}\left|\frac{\langle\varphi(t),f_n(t)-f(t)\rangle}{2}\right|<\frac{ \delta^2\varepsilon}3
	\end{equation*}
	 for all $n\geq N_2$.
	 
	Due to \eqref{eq:not_converge}, we can choose $\bar t$ with
	\begin{equation*}
	\left\|f_N(\bar t)-f(\bar t)\right\|_V\geq\varepsilon
	\end{equation*}
	for $N=\max(N_1,N_2)$. Now, it follows 
	\begin{align*}
	2\|f(\bar t)\|_V&\geq \|f(\bar t)\|_V+\|f_N(\bar t)\|_V - \Bigl|\|f_N(\bar t)\|_V-\|f(\bar t)\|_V\Bigr|\\
	&\geq\|f_N(\bar t)-f(\bar t)\|_V - \Bigl|\|f_N(\bar t)\|_V-\|f(\bar t)\|_V\Bigr|\\
	&>\varepsilon - \frac{\delta\varepsilon}3\\
	&\geq \frac{2\varepsilon}3.
	\end{align*}

	Thus, it is $\|f(\bar t)\|_V>\frac{\varepsilon}{3}$ and we conclude
	\begin{align*}
	\|f_N(\bar t)\|_V&\leq\Bigl|\|f_N(\bar t)\|_V-\|f(\bar t)\|_V\Bigr|+\|f(\bar t)\|_V\\
	&\leq \frac{\delta\varepsilon}{3}+\|f(\bar t)\|_V\\
	&< (1+\delta)\|f(\bar t)\|_V
	\end{align*}
	and furthermore, we compute
	\begin{equation*}
	\left\|\frac{f_N(\bar t)}{(1+\delta)\|f(\bar t)\|}-\frac{f(\bar t)}{(1+\delta)\|f(\bar t)\|}\right\|_V\geq\frac{\varepsilon}{(1+\delta)\|f(\bar t)\|_V}\geq \frac{\varepsilon}{2\|f\|_{C([0,T]),V)}}=\tilde{\varepsilon}.
	\end{equation*}
	Therefore, we can use the uniform convexity of $V$ to conclude
	\begin{align*}
	\left|\frac{\langle \varphi(\bar t),f_N(\bar t)+f(\bar t)\rangle}2\right|&\leq\frac{\|f_N(\bar t)+f(\bar t)\|_V}{2}\\
	&=(1+\delta)\|f(\bar t)\|_V\Biggl\|\frac{\frac{f_N(\bar t)}{(1+\delta)\|f(\bar t)\|}+\frac{f(\bar t)}{(1+\delta)\|f(\bar t)\|}}{2}\Biggr\|_V\\
	&\leq \|f(\bar t)\|_V(1+\delta)(1-\delta)\\
	&=\|f(\bar t)\|_V(1-\delta^2).
	\end{align*}
	Altogether, this yields the contradiction
	\begin{align*}
	\|f(\bar t)\|_V&=\langle \varphi(\bar t),f(\bar t)\rangle\\
	&\leq\left|\frac{\langle \varphi(\bar t),f_N(\bar t)+f(\bar t)\rangle}2-\langle\varphi(\bar t),f(\bar t)\rangle\right|+\left|\frac{\langle \varphi(\bar t),f_N(\bar t)+f(\bar t)\rangle}2\right|\\
	&<\frac{\delta^2\varepsilon}3+\|f(\bar t)\|_V(1-\delta^2)\\
	&<\|f(\bar t)\|_V\delta^2+\|f(\bar t)\|_V(1-\delta^2)\\
	&= \|f(\bar t)\|_V.
	\end{align*}
\end{proof}
\section{Conclusion}
In this note, we have seen how the Radon-Riesz property can be transferred to spaces of continuous functions. In contrast to the usual Radon-Riesz property in a Banach space, just the uniform convexity of the space is not sufficient to prove the property for the space of continuous functions. We need additionally the uniform smoothness. 

However, the aim of this note was not to find minimal requirements for this statement to be valid. Clearly, a necessary condition for the property in $C([0,T],V)$ is the Radon-Riesz property in $V$ itself since otherwise we could construct a counter example using constant functions. In particular, it is not sufficient to require the reflexivity of the space (see \cite{Bauschke.2003}).

	\bibliographystyle{plain}
	\bibliography{../bibliothek}

\end{document}